\documentclass[a4paper,11pt,reqno]{amsart} %{article} %{amsart}

\usepackage{latexsym}
\usepackage{amssymb,amsmath,amsopn}
\usepackage{amsthm} %apparently best to lead after amsmath
\usepackage{graphics}

\usepackage{setspace}

\newtheorem*{utheorem}{Theorem}
\newtheorem{theorem}{Theorem} %[section]

\newtheorem{lemma}[theorem]{Lemma}

\newtheorem{corollary}[theorem]{Corollary}

\newtheorem{proposition}[theorem]{Proposition}

\newcommand{\la}{\lambda}

\DeclareMathOperator{\sgn}{sgn}

\DeclareMathOperator{\End}{End}
\DeclareMathOperator{\Hom}{Hom}
\DeclareMathOperator{\rad}{rad}
\DeclareMathOperator{\soc}{soc}
\DeclareMathOperator{\Top}{top}

\newcommand{\lf}{\lfloor}
\newcommand{\rf}{\rfloor}

\newcommand{\N}{\mathbf{N}}

\newcommand{\Ind}{\big\uparrow}
\newcommand{\Res}{\big\downarrow}
\newcommand{\ind}{\!\!\uparrow}
\newcommand{\res}{\!\!\downarrow}

\newcommand{\mfrac}[2]{ { \textstyle{\frac{#1}{#2}} } }

\newcommand{\PGammaL}{\mathrm{P\Gamma L}}
\newcounter{thmlistcnt}
\newenvironment{thmlist}%
	{\setcounter{thmlistcnt}{0}%
	\begin{list}{\emph{(\alph{thmlistcnt})}}{%
		\usecounter{thmlistcnt}%
		\setlength{\topsep}{0pt}%
		\setlength{\leftmargin}{0pt}%
		\setlength{\itemsep}{0pt}% -3pt for article class with 1.175
		\setlength{\itemindent}{17pt}}%
	}%
	{\end{list}}%

\newcounter{romanthmlistcnt}
\newenvironment{romanthmlist}%
	{\setcounter{romanthmlistcnt}{0}%
	\begin{list}{\emph{(\roman{romanthmlistcnt})}}{%
		\usecounter{romanthmlistcnt}%
		\setlength{\topsep}{0pt}%
		\setlength{\leftmargin}{0pt}%
		\setlength{\itemsep}{0pt}% see above
		\setlength{\itemindent}{15pt}}%
	}%
	{\end{list}}%

\newcounter{romanlistcnt}
	{\setcounter{romanlistcnt}{0}%
	\begin{list}{(\roman{romanlistcnt})}{%
		\usecounter{romanlistcnt}%
		\setlength{\topsep}{0pt}%
		\setlength{\leftmargin}{17pt}%
		\setlength{\itemsep}{0pt}% see above
		\setlength{\itemindent}{15pt}}%
	}%
	{\end{list}}%

\begin{document}

\begin{abstract}
Building on work of Saxl, we classify 
the multiplicity-free permutation characters of all symmetric
groups of degree $66$ or more.
A corollary is
a complete list of the irreducible characters of symmetric groups
(again of degree $66$ or more)
which may appear in a multiplicity-free permutation representation.
The multiplicity-free characters
in a related family of monomial characters are also
classified. We end by investigating a consequence of these
results for Specht filtrations of permutation modules
defined over fields of prime characteristic.
\end{abstract}

\title{Multiplicity-free representations of symmetric groups}
\author{Mark Wildon}
\date{\today} 
\email{m.j.wildon@swan.ac.uk.}
\maketitle

\emph{Keywords:} multiplicity-free, symmetric group,
permutation character, monomial
character, Specht filtration

\emph{Email address:} \texttt{wildon@maths.ox.ac.uk}

\emph{Affiliation:} Mathematics Department, 
Swansea University, Singleton Park, Swansea, SA2 8PP

\emph{Address for correspondence:}
Department of Mathematics, 
University of Bristol, University Walk, Bristol, BS8 1TW.
\emph{Tel:} 07747 636959

\bigskip

%\section{Introduction}

In this paper we
 prove three theorems on the multiplicity-free
representations of symmetric groups. 
%the third concerns the behaviour of permutation representations
%of the symmetric group defined over fields of prime characteristic.
These theorems have  interesting consequences for the
permutation actions of symmetric groups, and for the
theory of Specht filtrations of permutation
modules, while also being of interest in their
own right.
	 
Our notation is standard.	
%To state these theorems, we must
%introduce some standard notation. 
Let~$S_n$ denote the symmetric
group of degree~$n$, and let~$\chi^\lambda$ denote the ordinary
irreducible character of $S_n$ canonically labelled by the partition~$\lambda$
 of $n$. 
 (For an account of the character
theory of the symmetric group see Fulton \& Harris \cite[Chapter 4]{FH},
or James~\cite{James}. We shall use James' 
lecture notes as the main source for the deeper results
we need.) 
A character $\pi$ of $S_n$ is said to be \emph{multiplicity-free}
 if $\left<\pi, \chi^\lambda\right> \le 1$ for all partitions~$\lambda$ of~$n$. 
If~$\theta$ is a character of a subgroup
of~$S_n$ then we write $\theta\ind^{S_n}$ %_H^{S_n}$ 
for the character of $S_n$ induced by~$\theta$. Dually, the arrow~$\ \res$ denotes restriction. Later we shall extend
this notation
from characters to their associated representations.
If~$H$ is a subgroup of $S_l$ then~$H \wr S_m$ is the wreath
product of $H$ with~$S_m$, acting as a subgroup of~$S_{lm}$ (as defined in~\cite[\S 1.10]{CameronPermGps}). Finally, let~$A_n$ denote the alternating group
of degree~$n$. 

Our first two theorems are motivated by a result of
Inglis, Richardson and Saxl \cite{IRS} which shows
that every irreducible character of a symmetric
group is a constituent of a multiplicity-free monomial
character.

\begin{utheorem}[Inglis, Richardson, Saxl]\label{thm:IRS}
Let $m \in \N$ and let $t$ be a fixed-point-free involution in the symmetric
group $S_{2m}$. If $n = 2m +f $ where $f \in \N_0$ then
\[ 1_{C_{S_{2m}}(t)} \times \sgn_{S_f} 
\Ind^{S_n}_{C_{S_{2m}}(t) \times S_f} 
= \sum_\lambda \chi^\lambda \]
where the
sum is over all partitions $\lambda$ of $n$ with precisely~$f$
odd parts. %\quad$\Box$
\end{utheorem}

Our Theorem~\ref{thm:IRSconverse} shows that conversely,
these characters 
are nearly the only
ones of their type that are multiplicity-free.

\begin{theorem}\label{thm:IRSconverse}
Let $n \geq 7$, let $k \le n$ and let $x \in S_k$ be a fixed-point-free
permutation. Let $\theta$ be a $1$-dimensional character of~$S_{n-k}$.
The monomial character
\[  \psi = \bigl( 1_{C_{S_k}(x)} \times \theta \bigr)
\Ind_{C_{S_k}(x) \times S_{n-k}}^{S_n}\]
%where $\theta$ is a $1$-dimensional character of~$S_{n-k}$. 
%The induced character
%$\psi\ind^{S_n}$ 
is multiplicity-free if and only if
\emph{either}~$x$ is a $2$-cycle and $k=2$ \emph{or}~$x$ is a $3$-cycle and $k=3$ \emph{or}
$x$ is a fixed-point-free involution in $S_k$ and $\theta = \sgn_{S_{n-k}}$.
\end{theorem}

We prove this theorem in \S 2 below. The proof is straightforward,
and will help to introduce the techniques used in the remainder
of the paper.

In light of the theorem of Inglis, Richardson and Saxl,
it is very natural to ask whether
every irreducible character of a symmetric group is a 
constituent of a multiplicity-free \emph{permutation} character.
Our second theorem, which builds on work of Saxl \cite{Saxl},
gives the classification needed to show that
this is very far from the case.

\begin{theorem}\label{thm:main}
Let $n \ge 66$.
The permutation character of $S_n$ acting on the cosets
of a subgroup $G$ is multiplicity-free if and only if one of:

\begin{thmlist}
\newcommand{\ep}{;}

\item[\emph{(a1)}]  $A_n \le G \le S_n$ or $A_{n-1} \le G \le S_{n-1}$ or
$G = A_{n-2} \times S_2$ or $G = S_{n-2} \times S_2$
or $G = A_n \cap (S_{n-2} \times S_2)$;

\item[\emph{(a2)}] $A_{n-k} \times A_k \le G \le S_{n-k} \times S_k$
where $3 \le k < (n-1)/2$;

\item[\emph{(b1)}] $n=2k$ and $A_k \times A_k < G \le S_k \wr S_2$\ep 

\item[\emph{(b2)}]
$n=2k+1$ and either $A_{k+1} \times A_k < G \le S_{k+1} \times S_k$
or $G$ is a subgroup of $S_k \wr S_2$ of index $\le 2$
other than $S_k \times S_k$\ep

\item[\emph{(c)}] $n = 2k$ or $n = 2k+1$ and $G = S_2 \wr S_k$\ep 

\item[\emph{(d)}] $n = 2k$ where $k$ is odd and $G = A_{2k} \cap (S_2 \wr S_k)$\ep

%and $G$ 
%is the index~$2$ subgroup of the wreath product
%$C_2 \wr S_k$ generated by
%the top group~$S_k$ together with the unique
%normal subgroup of $G$ of index~$2$ in the base group $C_2 \times \ldots
%\times C_2$\ep

\item[\emph{(e)}] $G = G_k \times A_{n-k}$ or $G = G_k \times S_{n-k}$
where~$k$ is either $5$, $6$ or $9$, or $G = A_n \cap (G_k \times S_{n-k})$
where $k = 5$ or $6$, and
%$G = A_{n-k} \times G_k$ or $G = S_{n-k} \times G_k$ 
%or $G = A_n \cap (S_{n-k} \times G_k)$ where~$k$ 
%is either $5$, $6$ or $9$, and 
$G_5$ is the Frobenius group of
order $20$ acting on $5$ points, $G_6$ is $\mathrm{PGL}(2,5)$ in its
natural projective action on $6$ points and
$G_9$ is $\PGammaL(2,8)$ in its natural projective
action on $9$ points.
\end{thmlist}
\end{theorem}

It seems unavoidable that
cases (a) and (b) have a slightly fiddly statement. The part 
of the proof that leads to these cases is, however, the most routine.
The reader may wish to refer ahead to Figure~2 in \S 2.2
which shows the subgroups  
of $S_k \wr S_2$ that appear in case (b).
%shows
%the full lattice of subgroups of $S_k \wr S_2$ which contain
%$A_k \times A_k$.

If $n=2k$ then the subgroup $S_2 \wr S_k$ in case~(c) is
the centralizer of a fixed-point-free involution
in $S_{n}$. It may also be helpful to recall
%To understand the subgroups in cases~(c)
%and~(d),
%it may perhaps be helpful
that $S_2 \wr S_k \le S_{2k}$ is permutation isomorphic 
to the Weyl group of type~$B_k$ in its
action on the vectors $\{ \pm \epsilon_1, \ldots,
\pm \epsilon_k \}$ 
spanning the root space for \hbox{$\mathsf{so}(2k+1, \mathbf{C})$}.
(See \cite[Chapter~3]{Humphreys} for more details.)
Under the natural embedding $\mathsf{so}(2k, \mathbf{C}) 
\rightarrow \mathsf{so}(2k+1,\mathbf{C})$,
the Weyl group of type~$D_k$ 
acts on $\{ \pm \epsilon_1, \ldots,
\pm \epsilon_k \}$ as a subgroup of index $2$ in~$B_k$;
it is this subgroup which appears in case~(d) of
Theorem~\ref{thm:main}. We define the subgroups in case~(e) 
more fully in \S 2.3 below.

A %n important 
corollary
%of the proof of this theorem 
of Theorem~\ref{thm:main} is
a complete list of the irreducible characters
of $S_n$ for $n \ge 66$  
which may appear in a multiplicity-free
permutation representation (see Corollary~\ref{cor} below).
The reader will see that there
are very few such characters;
in particular,
if~$n \ge 20$, then $\chi^{(n-4,3,1)}$
never appears in such a representation.\footnote{This result,
together with 
Theorem~\ref{thm:IRSconverse}, was first stated and proved  
in the author's D.~Phil thesis \cite[Chapter~4]{WildonDPhil}.}
Since a permutation character of a symmetric
group is multiplicity-free if and only if all the orbitals
in the corresponding permutation action are self-paired
(see \cite[page~46]{CameronPermGps}), Theorem~\ref{thm:main} also
serves to classify such actions. 

While Theorem~\ref{thm:main} is stated for $n \ge 66$,
the proof given in \S 2 below gives a complete classification
for all $n \ge 20$. The predicted list of subgroups has been checked
using the computer algebra
package {\sc magma} \cite{Magma} to be correct for $20 \le n \le 23$.
The same check has been made for the list
of irreducible characters in Corollary~\ref{cor}.
One could easily use {\sc magma} to generate the full
list of subgroups of symmetric groups of degree $< 20$
which give multiplicity-free permutation characters, but we shall not
pursue this possibility here. 
The author recently learned of parallel work by C.~Godsil
and K.~Meagher~\cite{Meagher}, to appear in Annals of
Combinatorics. Their paper gives a complete
classification of multiplicity-free
permutation characters of every degree. When $n \ge 20$,
their results are in agreement with Theorem~2.

Our third theorem concerns the permutation modules whose ordinary
characters appear in Theorem~\ref{thm:main}. Note that, by
\cite[Theorem~3.5]{CameronPermGps},
these are exactly
the permutation modules whose centralizer algebra is abelian.
%concerns the 
%structure of the permutation
%modules induced from the subgroups in Theorem~\ref{thm:main}.
To understand the statement of this theorem
the reader will need to know a little about
Specht modules: see \cite[Chapters~4,~5]{James} for an
introduction.
We recall here that if~$S_\mathbf{Z}^\lambda$
is the integral Specht module for~$\mathbf{Z}S_n$ 
labelled by the partition~$\lambda$ of~$n$
then, regarding  the entries of the representing matrices
as rational numbers,~$S_\mathbf{Z}^\lambda$
affords the ordinary irreducible character~$\chi^\lambda$. If,
instead, we regard the entries as elements of
a field $F$ of prime characteristic, then the resulting module 
for $FS_n$, denoted simply $S^\lambda$, is usually
no longer irreducible---indeed, determining the composition factors
of Specht modules in prime characteristic
is one of the main unsolved problems in modular representation theory.

We say that an $FS_n$-module $U$ has a \emph{Specht filtration}
if there is a chain of submodules
\[ 0 \subset U_1 \subset \cdots \subset U_r \subset U \]
such that each successive quotient is isomorphic to a Specht module.

\begin{theorem}\label{thm:mod}
Let $F$
be an algebraically closed
field of prime characteristic~$p > 3$ and let $n \ge 66$.
 If~$G$ is a subgroup of~$S_n$
such that the ordinary permutation character $1_G\ind^{S_n}$
is multiplicity-free, then each summand of the permutation module
\[ F \Ind_G^{S_n} \]
is a self-dual module with a Specht filtration.
The Specht module
$S^\lambda$ for~$FS_n$ appears in a Specht filtration
of $F \ind_G^{S_n}$
if and only if $\chi^\lambda$ is a
constituent of the ordinary character $1_G\ind^{S_n}$.
\end{theorem}

The author first suspected the existence of this theorem
after reading Paget's paper~\cite{PagetBrauer}. Paget's main 
result is that the permutation modules coming from
case (c) of Theorem~\ref{thm:main} have a Specht filtration, 
with the expected Specht factors.
It is a simple matter to adapt her work to deal with case~(d).

In \S 4 we show that if $F$ is an algebraically
closed field of characteristic~$3$ then
$F\ind_{\PGammaL(2,8)}^{S_9}$
does not have a Specht filtration. This gives an interesting
%, andalso fairly elementary, 
example
of a permutation module in odd prime characteristic
without a Specht filtration. 
Using more sophisticated
techniques, Mikaelian 
has
constructed a family of examples of such modules for fields of 
characteristic $p > 3$.\footnote{Personal communication, A.~Mikaelian,
Oxford, July 2007.}
The existence of such
modules is a clear indication that results such as Theorem~\ref{thm:mod}
cannot be obtained by any routine `reduction mod $p$' argument.

A preliminary investigation has shown that the situation in 
characteristic~$2$ is still more complicated. This
is to be expected on theoretical grounds:
see~\cite{HN} 
and~\cite{HemmerFiltrations} 
for an introduction to the general
theory of Specht filtrations.
To demonstrate the difficulties of working
in characteristic~$2$ we end \S 4 by showing 
that although the module $\mathbf{F}_2 \ind_{S_3 \wr S_2}^{S_6}$
has a Specht filtration, it \emph{does not} have a Specht filtration
with the Specht factors indicated by its ordinary character.

\section{Proof of Theorem~\ref{thm:IRSconverse}}
We first classify the multiplicity-free
permutation characters given by the actions of
symmetric groups on their conjugacy classes. For this 
we shall need the following lemma, ultimately
due to Frobenius,
which implies that multiplicity-free permutation characters only
come from permutation actions with relatively high degrees
of homogeneity. %The proof follows from equation~(2a) in~\cite{Saxl}.

\begin{lemma}\label{lemma:tworow}
Let $G \le S_n$ be a permutation group acting on $\{1,2,\ldots,n\}$.
Let~$\pi$ be the permutation character of the action
of $S_n$ on the cosets of~$G$. Let $t_r(G)$ be 
the number of orbits of $G$ on $r$-subsets of $\{1,2,\ldots,n\}$.
If~$0 \le r \le n/2$ then
\[ \bigl< \pi, \chi^{(n-r,r)} \big> =
\begin{cases} t_r(G) - t_{r-1}(G) & \text{if $r \ge 1$} \\
                1 & \text{if $r = 0$}
\end{cases} . \quad\Box
\]
\end{lemma}

We shall also need the forms of Young's rule and
Pieri's rule
given in the proposition below. Note that Pieri's rule follows
from Young's rule if we conjugate by the sign character,
so there is no need for us to use the Littlewood--Richardson rule.
(For a proof of Young's rule see \cite[Chapter~17]{James}.
The modular version of Young's rule proved by James in this reference
will be useful to us later---see Theorem~\ref{thm:JamesYoung}
in \S 3 below.)

\vbox{
\begin{proposition}\label{prop:Youngsrule}
Let $n \ge k \ge 1$ and let $\mu$ be a partition of $k$.
\begin{romanthmlist}
\item \emph{Young's rule:} $(\chi^\mu \times 1_{S_{n-k}}) \ind_{S_{k} \times S_{n-k}}^{S_n} = \sum \chi^\la$
where the sum is over all partitions~$\la$ obtained from~$\mu$ by
adding $n-k$ nodes, no two in the same column.

\item \emph{Pieri's rule:}
$(\chi^\mu \times \sgn_{S_{n-k}})\ind_{S_{k} \times S_{n-k}}^{S_n} = \sum \chi^\la$
where the sum is over all partitions~$\la$ obtained from~$\mu$ by
adding $n-k$ nodes, no two in the same \hbox{row.\quad$\Box$} %improves break
\end{romanthmlist}
\end{proposition}}

%It is interesting to note that 
%the theorem of Inglis, Richardson and Saxl stated in the introduction
%to this paper
% can be deduced from its better known
%special case when~$f=0$ by an application of Pieri's rule. 

\vbox{
\begin{proposition}\label{prop:conj}
Let $n \geq 7$, let $x \in S_n$ be a non-identity
permutation and let~$\pi$ be the permutation
character of $S_n$ acting on  the conjugacy
class of~$x$. 
Then~$\pi$ 
is multiplicity-free if and only if $x$ has one of the cycle types:

\begin{romanthmlist}
\item $(2,1^{n-2})$,

\item $(3,1^{n-3})$,

\item $(2^m)$ when $n = 2m$ or $n = 2m+1$.
\end{romanthmlist}

\noindent 
Furthermore, if $\pi$ is not multiplicity-free, then \emph{either} $\pi$
contains~$\chi^{(n-2,2)}$ more than once \emph{or} $x$ has cycle type 
$(3^m)$ where $n = 3m$ or $n=3m+1$.
\end{proposition}}

\medskip
\noindent\emph{Proof.}
%\begin{proof}
That $\pi$ is multiplicity-free in cases (i) and (ii) follows 
from Young's rule, while case (iii) is given by
the $f=0$ and $f=1$ cases of the theorem of Inglis, Richardson and Saxl.
(As Saxl notes in \cite{Saxl}, the $f=0$ case of this theorem dates back
at least to Thrall: see \cite[Theorem~III]{Thrall}.)

Now suppose that $\pi$ % = 1_{C_{S_n}(x)}\ind^{S_{n}}$
is multiplicity-free.
Applying
Lemma~\ref{lemma:tworow}
with the character $\chi^{(n-1,1)}$ shows that 
$t_1(C_{S_n}(x)) \le 2$, and hence
$C_{S_{n}}(x)$ has either~$1$ or~$2$ orbits on $\{1,\ldots,n\}$.
Similarly, applying Lemma~\ref{lemma:tworow} with
the character~$\chi^{(n-2,2)}$ shows that
\begin{equation}
\label{eq:pairs}
	t_2\bigl(C_{S_n}(x)\bigr) - t_1\bigl(C_{S_n}(x)\bigr) \le 1, 
\end{equation}
and hence $C_{S_n}(x)$ has at most $3$ orbits on the $2$-subsets
of $\{1,\ldots,n\}$.

Suppose first of all that $C_{S_{n}}(x)$ is transitive on $\{1,\ldots,n\}$.
 Then~$x$
must have cycle type~$(l^{m})$ for some $l \ge 2$ and $m \ge 1$ such that $n = lm$. The 
centralizer~$C_{S_n}(x)$ is permutation isomorphic to the wreath product
\hbox{$C_{l} \wr S_{m} \le S_n$}. It is not hard to see that the number of 
orbits of $C_l \wr S_m$ on unordered pairs from $\{1,\ldots,n\}$ is
\[ \begin{cases}  \lfloor l/2 \rfloor +1 & \text{if $m>1$} \\
					\lfloor l/2 \rfloor & \text{if $m=1$}
\end{cases}. \]
Comparing with \eqref{eq:pairs}, this
shows that if~$\pi$ is multiplicity-free then $l \leq 3$.

Now suppose that $C_{S_{n}}(x)$ has $2$ orbits on $\{1,\ldots,n\}$. 
The previous paragraph counts the number of orbits of $C_{S_n}(x)$ 
on unordered pairs
with both elements lying in a single orbit of $C_{S_n}(x)$ on $\{1,\ldots,n\}$.
It is clear that
there is  exactly one orbit involving unordered pairs of the
form $\left\{i,j\right\}$ with $i$ and $j$ taken from different
orbits of $C_{S_n}(x)$. We leave it to the reader to check
that these remarks imply that either
$n = 2m+1$ and $x$ has cycle type $(2^m,1)$,
or $n=3m+1$ and $x$ has cycle type $(3^m,1)$.

To finish the proof we must show that if $x$ has cycle type~$(3^m)$ 
or~$(3^m,1)$
then~$\pi$ is not multiplicity-free, 
even though it contains $\chi^{(n-2,2)}$ only once. 
The  simplest way to do this seems to be to count degrees.
Let $t_n$ be the sum of the degrees of all the
irreducible characters of  $S_n$.
We shall show that~$\pi(1) > t_n$ whenever $n \ge 12$. This
leaves only three cases to be analysed separately.

It follows from the theorem of Inglis, Richardson and Saxl that~$t_n$
is the number of elements of $S_n$ of order at most $2$ (of course this result can also be seen in other ways, for example via 
the Frobenius--Schur count of involutions, or the Robinson--Schensted
correspondence).
From this it follows that~$t_n = t_{n-1} + (n-1)t_{n-2}$ for $n \ge 2$
and hence that
$2t_{n-1} \le t_n \le nt_{n-1}$ for $n \ge 2$.
These results imply that
\[ t_{3(m+1)} = (6m+4)t_{3m} + 9m(m+1) t_{3m-1} \le \mfrac{1}{2}(9m^2+21m+8) t_{3m}.\]
Let $u_n = n!/ |C_{S_n}(x)|$ be the degree of $\pi$.
A short inductive argument using the last inequality
shows that $t_{3m} < u_{3m}$ for all $m \ge 4$.
Now, provided that $m \ge 4$, we have% may use this to get
\[ t_{3m+1} \le (3m+1)t_{3m} < (3m+1)u_{3m} = u_{3m+1} \]
which is the other inequality we require.

When $n = 10$, one finds that $\pi(1) = 22400$ and $t_{10} = 9496$, and so the
degree-counting approach also works in this case.
The remaining two cases can be checked by hand; one source for
the required character tables is \cite[Appendix~I.A]{JK}.
One finds that if $x$ has 
cycle type~$(3,3,1)$ then $\pi$ contains~$\chi^{(3,1^4)}$ twice, while
 if $x$ has cycle type $(3,3,3)$ then~$\pi$ contains both~$\chi^{(5,2,2)}$ 
 and $\chi^{(4,2,1,1,1)}$ twice.\quad$\Box$
%\end{proof}

\medskip
For $n \le 6$, one can show by direct calculation that if the permutation
character of~$S_n$ acting on the conjugacy class of a non-identity element~$x$
is multiplicity-free, then~$x$ has one  of the cycle types in the
table below. Note that if $n \le 4$ then all non-identity classes appear.

\begin{figure}[htp]
\begin{center}
\begin{tabular}{l|l}
$n$   & cycle types 
\\[1pt] \hline
$2$ &  $(2)$ \rule{0pt}{12pt}\\  
$3$ &  $(2,1)$, $(3)$ \\ 
$4$ & $(2,1^2)$, $(2,2)$, $(3,1)$, $(4)$\\
$5$ & $(2,1^3)$, $(2,2,1)$, $(3,1^2)$, $(3,2)$ \\
$6$ & $(2,1^4)$, $(2^3)$, $(3,1^3)$, $(3,3)$
\end{tabular}
\end{center}
\caption{Non-identity conjugacy classes of symmetric groups of
degree $\le 6$ whose associated permutation character
is multiplicity-free.}
\end{figure}

We are now ready to prove Theorem~\ref{thm:IRSconverse}. 
Let $n \ge 7$ and let $k \le n$.
Let $x \in S_k$ be a fixed-point-free permutation, let $\theta$
be a $1$-dimensional character of $S_{n-k}$, and let
$\psi = 
(1_{C_{S_k}(x)} \times \theta)\ind^{S_n}_{C_{S_k}(x) \times S_{n-k}}$. 
If $\theta$ is the trivial character then $\psi$ is merely
the permutation character of $S_n$ acting on the conjugacy
class of $S_n$ containing~$x$, so the result
follows from Proposition~\ref{prop:conj}.

We may therefore assume that $k < n$ and that $\theta = \sgn_{S_{n-k}}$.
Since
\[ \psi
= (1_{C_{S_k}(x)}\Ind^{S_k} \times \theta)\Ind_{S_k \times S_{n-k}}^{S_k}, \]
if $\psi$ is multiplicity-free, then $1_{C_{S_k}(x)}\ind^{S_k}$ must
also be multiplicity-free. 
If~$C_{S_k(x)}$ is not
transitive on $\{1,\ldots,k\}$ then we have seen that
\[ \bigl<  1_{C_{S_k}(x)}\Ind^{S_k}, \chi^{(k-1,1)} \bigr> \ge 1. \]
It now follows  from Pieri's rule that $\psi$ contains~$(k,1^{n-k})$ at least twice.
% (one
%copy comes from~$\chi^{(k-1,1)}$, the other from $\chi^{(k)}$).
Hence, $C_{S_n}(x)$ acts transitively, and by
Proposition~\ref{prop:conj} and the table above, 
either~$x$ is a fixed-point-free involution in $S_k$, or 
$x$ has cycle type $(3)$,~$(4)$ or~$(3^2)$ with $k = 3$,~$4$ or~$6$ respectively.

If $x$ is a fixed-point-free involution
then the theorem of Inglis, Richardson and Saxl states
that~$\psi$ is multiplicity-free. If $x$ is a $3$-cycle then
it follows from Pieri's rule that~$\pi$ is multiplicity-free.
If~$x$ is a $4$-cycle then
\[ \psi %= (1_{\left<(1234)\right>} \times \sgn_{S_{n-4}})\Ind^{S_n}
= \bigl( (\chi^{(4)} + \chi^{(2,2)} + \chi^{(2,1,1)}) \times 
\sgn_{S_{n-4}}\bigr) \Ind_{S_4 \times S_{n-4}}^{S_n}, \]
which contains $\chi^{(2,2,1^{n-4})}$ twice. Similarly, 
if~$x$ has cycle type~$(3^2)$ then
\[ \psi =
\bigl( (\chi^{(6)} + \chi^{(4,2)} + \chi^{(4,1,1)} + \chi^{(3,1^3)}
+ \chi^{(2,2,2)} + \chi^{(2,1^4)}) \times \sgn_{S_{n-6}} \bigr)
\Ind_{S_6 \times S_{n-6}}^{S_n}, \]
which contains $\chi^{(4,1^{n-4})}$ twice. This completes the proof 
of Theorem~\ref{thm:IRSconverse}.
%The remaining case may be despatched by similar arguments.

\section{Proof of Theorem~\ref{thm:main}}
%\subsection{} 
A very large step towards classifying the multiplicity-free
permutation characters of symmetric groups
was made by Saxl in \cite{Saxl}. In this paper Saxl gives
a list of subgroups of $S_n$ for $n \ge 19$, which he proves
contains all subgroups $G$ such that the permutation character
of $S_n$ acting on the cosets of $G$ is multiplicity-free.
Our contribution is to prune his list of the unwanted subgroups.
There are several interesting features
that still remain for us to discover, and  to obtain
the most uniform result, we must assume that~$n \ge 66$.

%The proof given below omits 
%several routine arguments
%in the interests of brevity.
Since we shall frequently need to refer to it, we give a verbatim
statement of Saxl's theorem from \cite[page 340]{Saxl}. 
There is a minor error in case (v), to which the groups
$A_n \cap (S_{n-k} \times G_k)$ for $k= 5,6$ should be added.
(It follows from the argument at the bottom of page~342 
of Saxl's paper  that
these groups should be considered for inclusion, and in fact 
both give
rise to multiplicity-free characters.)
%The author
%found this error in the course of verifying Theorem~\ref{thm:main}
%for $20 \le n \le 23$ using the computer algebra package {\sc magma}

\vbox{
\begin{utheorem}[Saxl]
Let $S_n$ be multiplicity-free on the set of cosets 
\emph{[denoted~$\Omega$]} of a
subgroup $G$. Assume that $n > 18$. Then one of:
\begin{romanthmlist}
\item $A_{n-k} \times A_k \le G \le S_{n-k} \times S_k$ for
some $k$ with $0 \le k < n/2$;
\item $n=2k$ and $A_k \times A_k < G \le S_k \wr S_2$;
\item $n=2k$ and $G \le S_2 \wr S_k$ of index at most $4$;
\item $n=2k+1$ and $G$ fixes a point of $\Omega$ and is one of the groups in \emph{(ii)} or \emph{(iii)} on the rest of $\Omega$; or
\item $A_{n-k} \times G_k \le G \le S_{n-k} \times G_k$ where $k$ is $5$, $6$
or $9$ and $G_k$ is Frobenius of order $20$, 
$\mathrm{PGL}(2,5)$ or $\PGammaL(2,8)$
respectively.
\end{romanthmlist}
\end{utheorem}}
%
%Before embarking on a case-by-case analysis, we mention that
%the corrected form of Saxl's Theorem, and our Theorem~\ref{thm:main}, 
%have been verified using the computer algebra package {\sc magma}
%for $20 \le n \le 23$.

We now examine each of Saxl's cases in turn. 
The most interesting
case~(iii) is left to the end, and 
we consider case~(iv) together with~(ii) and~(iii).
We shall frequently need the well-known
result (see for example \cite[6.6]{James}) 
that if $\lambda$ is a partition then
\begin{equation}
\label{eq:chartwist}
\chi^\lambda \times \sgn = \chi^{\lambda'} 
\end{equation} 
where $\lambda'$
is the conjugate partition to $\lambda$. (Recall that $\lambda'$
is the partition
defined by $\lambda'_i = |\{ j : \lambda_j \ge i\}|$; the
diagram of~$\lambda'$ is obtained from the diagram of~$\lambda$ by
reflecting it in its main diagonal.)
We shall also frequently use the fact that if $H < G < S_n$
and $1_H\ind^{S_n}$ is multiplicity-free, then $1_G\ind^{S_n}$
is also multiplicity-free.

\subsection{Case (i)}
If $k = 0$ or $k=1$ then  the subgroups
from this case clearly give multiplicity-free characters. 
They contribute to our case~(a1). 
If \hbox{$2 \le k \le n/2$} then
it follows from \eqref{eq:chartwist} together with
 Young's rule and Pieri's rule that
\begin{align} 1_{A_{n-k} \times A_k}\Ind^{S_n}
&= (1_{S_{n-k} \times S_k}\Ind^{S_n} + \sgn_{S_{n-k}} \times 1_{S_k}\Ind^{S_n})
(1 + \sgn_{S_n})\notag\\
%\\&\qquad\qquad+ 1_{S_{n-k}} \times \sgn_{S_k} \Ind^{S_n} 
%+ \sgn_{S_{n-k} \times S_k}\Ind^{S_n} \\
&= \sum_{i=0}^k \chi^{(n-i,i)} + \sum_{i=0}^k \chi^{(n-i,i)'} 
 + \chi^{(n-k,1^k)} 
+ \chi^{(n-k,1^k)'}
\notag\\ &\qquad\qquad +\chi^{(n-k+1,1^{k-1})} + \chi^{(n-k+1,1^{k-1})'}.
\label{eq:altalt} 
\end{align}
Hence, for $k$
in this range, $1_{A_{n-k} \times A_k}\ind^{S_n}$ is
multiplicity-free unless $k=2$ or $k = \lfloor n/2 \rfloor$.
%Thus in the unexceptional cases, if$A_{n-k} \times A_k < G \le S_{n-k}
%\times S_k$ then $1_G\ind^{S_n}$ is also multiplicity free.
When $k=2$ it is easily seen that $1_{S_{n-2}}\ind^{S_n}$
is not multiplicity-free, while if $G$ is one of the other two
index $2$-subgroups of $S_{n-2} \times S_2$, namely $A_{n-2} \times S_2$
or $A_n \cap (S_{n-2} \times S_2)$, then $1_G\ind^{S_n}$ \emph{is} 
multiplicity-free.
This gives the remaining groups in our case (a1) and 
the groups in case (a2).

If $n$ is even
then we have already dealt with all the groups from Saxl's case (i).
If $n=2k+1$ is odd then we still have to deal with the
subgroups of $S_{k+1} \times S_k$ properly containing $A_{k+1} \times A_k$.
A calculation similar to \eqref{eq:altalt}
shows that all these groups give multiplicity-free characters; they
appear in our case~(b2).

\subsection{Case (ii)} 
There are
three index $2$ subgroups of $S_k \wr S_2$, namely
$S_k \times S_k$, $A_{2k} \cap (S_k \wr S_2)$ and one other,
which we shall denote by $\Gamma_k$. 
Figure~2 overleaf shows the lattice  of subgroups we
must consider; note that they are in bijection
with the subgroups of the dihedral group of order~$8$.

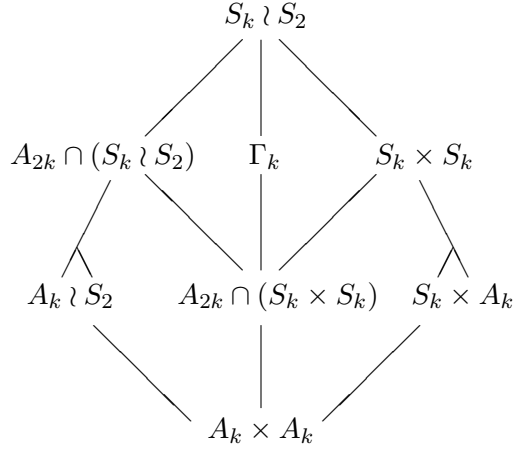
\begin{figure}[!h]\label{fig:Sk2}
\begin{center}
\setlength{\unitlength}{0.22cm}
\begin{picture}(16,26)
\put(9.1,10.1){\line(1,1){5.8}}
\put(6.9,10.1){\line(-1,1){5.8}}
\put(8,10.1){\line(0,1){5.8}}
\put(6.9,24){\line(-1,-1){5.8}}
\put(8,24){\line(0,-1){5.8}}
\put(9.1,24){\line(1,-1){5.8}}

\put(8,6.8){\line(0,-1){5}}
\put(-1,15.5){\line(-1,-2){2.9}}
\put(-3,11.5){\line(1,-2){0.9}}
\put(17.5,15.5){\line(1,-2){2.9}}
\put(19.5,11.5){\line(-1,-2){0.9}}

\put(-2,6.8){\line(1,-1){5.8}}
\put(17.5,6.8){\line(-1,-1){5.8}}

\put(3,8.1){$A_{2k} \cap (S_k \times S_k)$}
\put(14.8,16.5){$S_k \times S_k$}
\put(7.275,16.5){$\Gamma_k$}
\put(-7,16.5){$A_{2k} \cap (S_k \wr S_2)$}
\put(5.8,25){$S_k \wr S_2$}
\put(-6,8.1){$A_k \wr S_2$}
\put(17,8.1){$S_k \times A_k$}
\put(4.8,0){$A_k \times A_k$}

\end{picture}
\end{center}

\caption{Subgroups of $S_k \wr S_2$ containing $A_k \times A_k$ when
$k$ is even. If $k$ is odd then the 
 labels $A_{2k} \cap (S_k \wr S_2)$ and $\Gamma_k$
should be swapped. The forked lines to the subgroups
$S_k \times A_k$ and $A_k \wr S_2$ indicate that they appear
in two conjugate copies.}
\end{figure}

From \eqref{eq:altalt} we know that $1_{A_k \times A_k}\ind^{S_{2k}}$
is not multiplicity-free. However, it turns out that
every subgroup of $S_k \wr S_2$ which properly
contains $A_k \times A_k$ does give
a multiplicity-free permutation character.
These groups appear in our case (b1).
For later use we give their permutation characters
in full. 
We shall need the
character $\psi_k$ of $S_k \wr S_2$
defined by the composition of maps
$S_k \wr S_2 \twoheadrightarrow S_2
%\twoheadrightarrow 
\cong \{ \pm 1 \}$; note that $\Gamma_k = \ker \psi_k\sgn_{S_k \wr S_2}$.
Example~2.3 in~\cite{Saxl} tells us that
\begin{equation}\label{eq:twowreath}
1_{S_k \wr S_2}\Ind^{S_{2k}} = \sum_{i=0}^{\lf k/2 \rf}\chi^{(2k-2i,2i)}. 
\end{equation}
Given \eqref{eq:twowreath}, it follows
 from the known decomposition of $1_{S_k \times S_k}\ind^{S_{2k}}$
that
\begin{equation}
\label{eq:psi} \psi_k \Ind_{S_k\wr S_2}^{S_{2k}} = 
\sum_{i=0}^{\lf (k-1)/2 \rf}\chi^{(2k-2i-1,2i+1)}. 
\end{equation}
Using  \eqref{eq:chartwist} and \eqref{eq:psi} together with Young's rule and Pieri's rule 
we find that
\begin{align*}
%1_{S_k \wr S_2}\Ind^{S_{2k}} &= \sum_{i=0}^{k/2}\chi^{(2k-2i,2i)} \\
%1_{S_k \times S_k}\Ind^{S_{2k}} &= \sum_{i=0}^{k}\chi^{(2k-i,i)}, \\
1_{A_{2k} \cap (S_k \wr S_2)}\Ind^{S_{2k}} &=
1_{S_k \wr S_2}\Ind^{S_{2k}} +  \sgn_{S_k \wr S_2}\Ind^{S_{2k}} \\ &\qquad =
\sum_{i=0}^{\lf k/2 \rf}\chi^{(2k-2i,2i)} 
+ \sum_{i=0}^{\lf k/2 \rf} \chi^{(2k-2i,2i)'}, \\
1_{\Gamma_k} \Ind^{S_{2k}} &= 1_{S_k \wr S_2}\Ind^{S_{2k}} + 
\psi_k\sgn_{S_k \wr S_2}\Ind^{S_{2k}} \\ &\qquad = 
\sum_{i=0}^{\lf k/2 \rf}\chi^{(2k-2i,2i)} + 
\sum_{i=0}^{\lf (k-1)/2 \rf}\chi^{(2k-2i-1,2i+1)'}.
\end{align*}

Similar calculations give the permutation characters induced
from the index $4$ subgroups:
\begin{align*}
1_{S_k \times A_k}\Ind^{S_{2k}} &= \sum_{i=0}^{k}\chi^{(2k-i,i)}
+ \chi^{(k+1,1^{k-1})} + \chi^{(k,1^{k})}, \\
1_{A_{2k} \cap (S_k \times S_k)}\Ind^{S_{2k}} &=\sum_{i=0}^{k}\chi^{(2k-i,i)}
+ \sum_{i=0}^{k}\chi^{(2k-i,i)'}, \\
1_{A_k \wr S_2}\Ind^{S_{2k}} &=\sum_{i=0}^{\lf k/2 \rf} \chi^{(2k-2i,2i)} +
\chi^{(k+1,1^{k-1})} + \chi^{(k,1^k)} + \alpha_k, 
\end{align*}
where in the last line
%don't need lfloor etc here
\[ \alpha_k = \begin{cases} \sum_{i=0}^{k/2}\chi^{(2k-2i,2i)'} & \text{if $k$ is even} \\ \sum_{i=0}^{(k-1)/2}\chi^{(2k-2i-1,2i+1)'} & \text{if $k$ is odd.}
\end{cases} \]

To decide which of
these characters remain multiplicity-free
when induced from $S_{2k}$ to $S_{2k+1}$, and so should be taken from
Saxl's case (iv),
we first note that
\[ 1_{S_k \times S_k}\Ind^{S_{2k+1}} = \left( \sum_{r=0}^k \chi^{(2k-r,r)}\right)\Ind^{S_{2k+1}} \]
contains $\chi^{(2k,1)}$ twice. (In the second induction above,
Young's rule may be replaced with the ordinary branching rule: see
\cite[Chapter~9]{James}.)
Hence, if~$G$ is any subgroup of
$S_k \times S_k$, then $1_G\ind^{S_{2k+1}}$ is not multiplicity-free.
Similarly one shows that the permutation character induced from
$A_k \wr S_2$ is not multiplicity-free, while the characters induced from
$A_{2k} \cap (S_k \wr S_2)$ and~$\Gamma_k$ are. This gives
the remaining groups in our case (b2).

%\subsection{}
\subsection{Case (v)}
We now turn to Saxl's case (v). 
The subgroups $G_k$ for $k = 5, 6, 9$
are each $\left\lfloor k/2 \right\rfloor$-homogeneous.  
(It follows from Young's rule and Lemma~4 that
this is a necessary condition
for the induced characters $1_{G_k \times S_{n-k}}\ind^{S_n}$ to be
multiplicity-free for every $n$.) 
They are:
the~$2$-transitive Frobenius group $G_5 = \left<(12345), (2354)\right> \le S_5$; 
the $3$-transitive subgroup~$G_6 = \mathrm{PGL}(2,5) \le S_6$; and
the $3$-transitive but $4$-homogeneous subgroup~$G_9 = \PGammaL(2,8) \le S_9$. 
Here $\PGammaL(2,8)$ denotes the split extension of $\mathrm{PGL}(2,8)$
given by the order~$3$ Frobenius twist $F : \mathbf{F}_8 \rightarrow
\mathbf{F}_8$.
%(For $P{\mathit\Gamma} L(2,8)$ 
%the computer algebra package {\sc magma} was used.) When $n \ge 20$ they are
%all multiplicity-free. 

Calculation using Young's rule shows that the permutation characters
\begin{align*}
1_{G_5} \times 1_{S_{n-5}} \Ind^{S_n} &= \bigl( \chi^{(5)} + \chi^{(2,2,1)} \bigr) 
\times 1_{S_{n-5}} \Ind_{S_5 \times S_{n-5}}^{S_n}, \\
1_{G_6} \times 1_{S_{n-6}} \Ind^{S_n} &= \bigl( \chi^{(6)} + \chi^{(2,2,2)} \bigr)
\times 1_{S_{n-6}} \Ind_{S_6 \times S_{n-6}}^{S_n}
\end{align*}
are always multiplicity-free. A nice way to obtain these equations
uses the outer automorphism of $S_6$: if $H \cong S_5$ is 
a point stabiliser in $S_6$, then~$H$ is mapped under an outer
automorphism of $S_6$ to a subgroup permutation isomorphic to 
$G_6 = \mathrm{PGL}(2,5)$.
Inspection of the character table of $S_6$ shows that the 
constituent $\chi^{(5,1)}$ of $1_H\ind^{S_6}$ is mapped to the 
constituent~$\chi^{(2,2,2)}$ of~$1_{G_5}\ind^{S_6}$.
Since $G_6 \cap H$ is conjugate in $S_5$ to $G_5$, the
character induced from~$G_5$ can then be obtained by
restriction.

The remaining character from Saxl's case (v) is
\begin{equation}\label{eq:PGammaL} 
1_{G_9} \times 1_{S_{n-9}} \Ind^{S_n} \! = \! 
\bigl( \chi^{(9)} + \chi^{(1^9)} + \chi^{(5,1^4)} +
\chi^{(4,4,1)} + \chi^{(3,2^3)} \bigr) \times 1_{S_{n-9}} \Ind_{S_9 \times S_{n-9}}^{S_n}, 
\end{equation}
which is  always multiplicity-free.
%The easiest way
%to obtain this equation is probably to use a computer algebra
%package such as {\sc magma}. 
We outline one way to obtain this equation. One easily
checks that $\PGammaL(2,8)  \le A_9$, so by~\eqref{eq:chartwist}, 
$\chi^\lambda$ appears in 
$\pi = 1_{\PGammaL(2,8)}\ind^{S_9}$ if and only if $\chi^{\lambda'}$
appears. By
Lemma~\ref{lemma:tworow}, none of 
$\chi^{(8,1)}, \chi^{(7,2)}$, $\chi^{(6,3)}, \chi^{(5,4)}$, or their
conjugates,
appears in $\pi$. From
the equation 
\begin{equation*} %\label{eq:hook}
1_{S_{n-r}} \times \sgn_{S_r} \Ind^{S_n} = \chi^{(n-r,1^r)} +
\chi^{(n-r+1,1^{r-1})} \quad \text{if $1 \le r < n$}
\end{equation*}
and Frobenius reciprocity one sees that
%By Lemma~\ref{lemma:hook},
\[ \left< \pi, %\left< 1_{\PGammaL(2,8)}\Ind^{S_9},
 \chi^{(9-r,1^r)} \right> = \rho_r - \rho_{r-1}
+ \cdots + (-1)^r \rho_0 \quad \text{if $0 \le r < 9,$}\]
where $\rho_i = 1$ if $\PGammaL(2,8) \cap (S_{9-r} \times S_r) \le S_{9-r} \times
A_r$, and $\rho_i = 0$ otherwise. 
(Since $\PGammaL(2,8)$ is
$4$-homogeneous, it does not matter which subgroup $S_{9-r} \times S_r
\le S_9$ we choose.) Clearly $\rho_0 = \rho_1 = 1$.
The group $\mathrm{PGL}(2,8)$ has a unique conjugacy
class of elements of even order; these are involutions, and since
$\mathrm{PGL}(2,8)$ is sharply $3$-transitive, they must act 
%on the projective line over $\mathbf{F}_8$ 
with cycle type~$(2^4,1)$. Hence
$\rho_2 = \rho_3 = 0$. 
It follows from the identity
$(Fg)^3 = g^{F^2}g^Fg$ for $g \in \mathrm{PGL}(2,8)$, that
%\[ (Fg)^3 = g^{F^2}g^Fg \quad \text{for $g \in \mathrm{PGL(2,8)}$}, \]
the only new even order that appears when we extend
$\mathrm{PGL}(2,8)$ to $\PGammaL(2,8)$ is~$6$. 
Therefore no $4$-cycles
appear in the cycle decomposition of any element of $\PGammaL(2,8)$,
and %hence 
$\rho_4 = 1$. 
Hence, apart from~$\chi^{(9)}$ and $\chi^{(1^9)}$,
the only hook character to appear in $\pi$ is
$\chi^{(5,1^4)}$. We now have
\[ \pi = \chi^{(9)} + \chi^{(1^9)} + \chi^{(5,1^4)} + \psi\]
where $\psi$ has degree $168$.  With the exception of $\chi^{(3,3,3)}$
(which has degree~$42$)
and the pair $\chi^{(4,4,1)}$, $\chi^{(3,2^3)}$ (each of degree $84$), all the 
irreducible characters of $S_9$ that are still eligible to appear 
in~$\psi$ have
too high a degree. If~$\chi^{(3,3,3)}$ appears
four times, then we would have $\pi\bigl((12345)\bigr) = 10$;
the required character values may be computed by hand, or found 
in \cite[Appendix~I.A]{JK}.
However, $\PGammaL(2,8)$ contains no elements of order $5$, so clearly
$\pi((12345))= 0$.
Equation~\eqref{eq:PGammaL} follows.

It is straightforward to check using the formulae
\begin{align*} 
1_{G_k} \times 1_{A_{n-k}} \Ind^{S_n} &= \left(1_{G_k} \times 1_{S_{n-k}}
\right)\Ind^{S_n} + \left(
1_{G_k} \times \sgn_{S_{n-k}}\right)\Ind^{S_n}, \\
1_{A_n \cap (G_k \times S_{n-k})} \Ind^{S_n} &= 
\left(1_{G_k} \times 1_{S_{n-k}}\right) \Ind^{S_n} +
\left(1_{G_k} \times 1_{S_{n-k}}\right) \Ind^{S_n} \times \sgn_{S_n}
\end{align*}
and  Pieri's rule
that, provided $n\ge 20$, the
characters $1_{G_k} \times 1_{A_{n-k}} \ind^{S_n}$ and
$1_{A_n \cap (G_k \times S_{n-k})} \ind^{S_n}$ are also multiplicity-free.
This gives the groups in our case (e).

\subsection{Case (iii)}
It remains to deal with Saxl's case (iii): subgroups of \hbox{$S_2 \wr S_k$} of
index at most~$4$. By the theorem of Inglis, Richardson and Saxl,
$1_{S_2 \wr S_k}\!\ind^{S_{2k}}$ is
multiplicity-free. Moreover, this character is still multiplicity-free
if we induce up to $S_{2k+1}$, since
\[ 1_{S_2 \wr S_k} \Ind^{S_{2k+1}} = \sum_\lambda \chi^\lambda \]
where the sum is over all partitions $\lambda$ of $2k+1$ with exactly
one odd part. We therefore take $S_2 \wr S_k$ from Saxl's case (iv).
This gives our case (c).
It now only
remains to look at the proper subgroups
of~$S_2 \wr S_k$. 

Let $H$ be the unique
normal subgroup of $S_2 \wr S_k$ of index~$2$ 
in the base group $S_2 \times \cdots \times S_2$.
%(More concretely, if we choose generators
%for the base group $t_1, \ldots, t_n$, so that
%$t_i^\sigma = t_{i\sigma}$ for $\sigma$ in the top group,
%then $H = \left< t_i t_{i+1} : 1 \le i < n \right>$.)
A straightforward argument shows that,
provided $k \ge 5$, 
the group~$H \rtimes A_k$
is the only subgroup of $S_2 \wr S_k$ of index~$4$.
This subgroup is normal in $S_2 \wr S_k$,
and the quotient group $S_2 \wr S_k \,\bigl/ H \! \rtimes \! A_k$
is isomorphic to $C_2 \times C_2$. 
It follows that there are three subgroups of index~$2$ 
in $S_2 \wr S_k$, namely $H \rtimes S_k$, $S_2 \wr A_k$, and
one other, which we shall denote by $\Delta_k$. 
The subgroup
lattice is shown in Figure~3 below.

\bigskip

\begin{figure}[h!]\label{fig:S2k}
\begin{center}
\setlength{\unitlength}{0.22cm}
\begin{picture}(16,17)
\put(9.1,1.1){\line(1,1){5.8}}
\put(6.9,1.1){\line(-1,1){5.8}}
\put(8,1.1){\line(0,1){5.8}}
\put(6.9,15){\line(-1,-1){5.8}}
\put(8,15){\line(0,-1){5.8}}
\put(9.1,15){\line(1,-1){5.8}}

\put(5.8,-0.9){$H \rtimes A_k$}
\put(14.8,7.5){$S_2 \wr A_k$}
\put(7.275,7.5){$\Delta_k$}
\put(-3.75,7.5){$H \rtimes S_k$}
\put(5.8,16){$S_2 \wr S_k$}
\end{picture}
\end{center}
%\vspace{-0.2cm}
\caption{Subgroups of index at most $4$ in $S_2 \wr S_k$}
\end{figure}
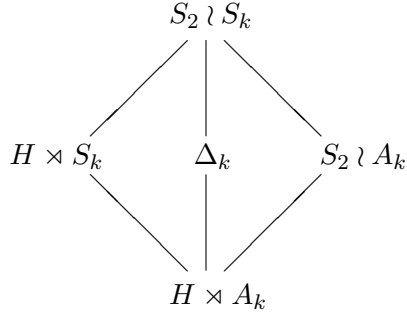

\bigskip

It is easy
to check that the
subgroup $H \rtimes S_k$ is equal to $A_{2k} \cap (S_2 \wr S_k)$.
Hence 
%$1_{H \rtimes S_k}\ind^{S_2 \wr S_k} = 1_{S_2 \wr S_k} + \sgn\res^{S_{2n}}_{S_2 \wr S_k}$.
%It follows that
\[
1_{H \rtimes S_k}\Ind^{S_{2k}} = 1_{S_2 \wr S_k}\Ind^{S_{2k}} + 
\sgn_{S_2 \wr S_k}\Ind^{S_{2k}}
\]
and so
\begin{equation}\label{eq:collide} 
1_{H \rtimes S_k}\Ind^{S_{2k}} = \sum \chi^\lambda + \sum \chi^{\lambda'}
\end{equation}
where the sums are over all partitions $\lambda$ of $2k$ with only even parts.

From now on, we shall say that a partition all of whose parts
are even is \emph{even}.
We see from \eqref{eq:collide} that $1_{H \rtimes S_k}\ind^{S_{2k}}$ fails
to be multiplicity-free if and only if there is an even partition~$\lambda$
whose conjugate~$\lambda'$ is also even. If $k$ is even then~$(k,k)$
is such a partition, while if $k$ is odd then it is clear
that no such partition can exist. This gives case (d) of 
Theorem~\ref{thm:main}.

Suppose that $k$ is odd. If we induce the character 
$1_{H \rtimes S_k}\ind^{S_{2k}}$ up to~$S_{2k+1}$,
then we obtain the constituent $\chi^{(k+1,k)}$ twice:
once by adding a node to the even partition $(k+1,k-1)$, and
once by adding a node to the partition~$(k,k)$, whose conjugate~$(2^n)$
is even. The group $H \rtimes S_k$ is therefore not included
in those coming from Saxl's case (iv).

To complete the proof of Theorem~\ref{thm:main}, it 
suffices to show that if~\hbox{$k \ge 33$}, then 
neither of the permutation characters induced from
the other two index~$2$ subgroups of $S_2 \wr S_k$ is multiplicity-free.
%We first consider~$C_2 \wr A_k$. 
In order to 
describe the constituents of these permutation characters
% $1_{C_2 \wr A_k}\ind^{S_{2k}}$
we shall use the following notation: if 
$\alpha = (a_1, a_2, \dots , a_r)$ is a partition of $k$
with distinct parts, let 
$2[\alpha] = 2[a_1, \ldots, a_r]$ denote the partition $\la$ of $2k$ whose 
leading diagonal hook
lengths are $2a_1, \ldots, 2a_r$, and such that 
$\la_i = a_i+i$ for $1 \le i \le r$. 
 For
instance, Figure 4 overleaf shows
 $2[4,3,1]$.

\begin{figure}[!h]
\begin{center}
%\vspace{-0.5cm}
\includegraphics{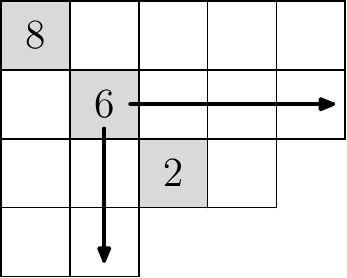}
\end{center}
%\vspace{-0.5cm}
\caption{Leading diagonal hook lengths in $2[4,3,1] = (5,5,4,2)$}
\end{figure}

%\newpage

We can now state the following lemma, which is the analogue of
\eqref{eq:psi} in~\S 2.2.

\vbox{
\begin{lemma}\label{Lemma:TwistedPerm}
Let $k \ge 2$ and let $\theta_k$ be the character of $S_2 \wr S_k$
defined by the composition of maps $S_2 \wr S_k \twoheadrightarrow S_k
\stackrel{\sgn}{\twoheadrightarrow} \{ \pm 1 \}$.
Then $1_{S_2 \wr A_k}\ind^{S_2 \wr S_k} = 1_{S_2 \wr S_k} + \theta_k$ and 
\[ \theta_{k}\Ind^{S_{2k}} = \sum \chi^{2[\alpha]} \]
where the sum is over 
all partitions $\alpha$ of $k$ with
distinct parts.
\end{lemma}}

Before proving Lemma~\ref{Lemma:TwistedPerm} we use it to complete the proof 
of Theorem~\ref{thm:main}. By the first statement in the lemma we have
\[ 1_{S_2 \wr A_k}\Ind^{S_{2k}} = 1_{S_2 \wr S_k}\Ind^{S_{2k}} + 
\theta_k\Ind^{S_{2k}}.\]
Hence, $1_{S_2 \wr A_k}\ind^{S_{2k}}$ fails to be multiplicity-free
if and only if there is an even partition of the form $2[\alpha]$.
Now, 
the partition $2[a_1, \ldots, a_r]$ is even if and only if the 
following conditions hold:
$a_{2i-1}$ is odd and
$a_{2i} = a_{2i-1}-1$ for all $i \le r/2$, and,
if~$r$ is odd, then~$a_r = 1$. It follows, on setting $2b_i = a_{2i-1}-1$, that
$1_{S_2 \wr A_k}\ind^{S_{2k}}$ fails to be multiplicity-free if
and only if there is a strictly decreasing sequence of positive
integers $(b_1, \ldots, b_s)$ such that either
\[ k = 4\sum_{i=1}^s b_i + s \quad \hbox{or} \quad k = 4\sum_{i=1}^s b_i + s + 1. \]
One now shows, by looking
at the possible values of $k$ mod $4$, that provided~$k \ge 25$, at least one 
of these equations has a solution. For example, if~$k = 4l$ with $l \ge 7$,
then one can solve the second equation by taking 
$s = 3$, $b_1 = l-4$, $b_2 = 2$ and $b_3 = 1$.
The bound on $k$ is strict: when $k = 24$, neither
equation is soluble, and hence the permutation character
$1_{S_2 \wr A_{24}}\ind^{S_{48}}$ is multiplicity-free.

Finally we consider the subgroup $\Delta_k$. It is easy to check that
\[ 1_{\Delta_k}\Ind^{S_2 \wr S_k} = 1_{S_2 \wr S_k} + \theta_k \sgn_{S_2 \wr S_k}. \]
Hence
\[ 1_{\Delta_k}\Ind^{S_{2k}} = 1_{S_2 \wr S_k}\Ind^{S_{2k}} + 
\theta_k\Ind^{S_{2k}} \times \sgn_{S_{2k}}.\]
By \eqref{eq:chartwist} and Lemma~\ref{Lemma:TwistedPerm}, we see that
$1_{\Delta_k}\ind^{S_{2k}}$ fails to be
multiplicity-free if and only there is a partition
$2[\alpha]$ whose \emph{conjugate} is even.
The partition $2[a_1, \ldots, a_r]$ has an even conjugate
if and only if $a_{2i-1}$ is even and
$a_{2i} = a_{2i-1}-1$ for all $i \le r/2$, and $r$ is even.
It follows, on setting $2b_i = a_{2i-1}$, that $1_{\Delta_k}\ind^{S_{2k}}$
fails to be multiplicity-free
if and only if there is a strictly decreasing sequence of positive integers
$(b_1,\ldots, b_s)$ such that 
\[ k = 4\sum_{i=1}^s b_i - s .\]
By a very similar argument to before, we now find
that $1_{\Delta_k}\ind^{S_{2k}}$ 
is not multiplicity-free if $k \ge 33$.
Again this bound is strict. 

\subsection*{Proof of Lemma \ref{Lemma:TwistedPerm}}
It is easy to see that $\theta_k$ is the unique non-trivial constituent
of $1_{S_2 \wr A_k}\ind^{S_2 \wr S_k}$. To proceed further,
we adapt the proof 
of the decomposition of $1_{S_2 \wr S_k}\ind^{S_{2k}}$
attributed to James and Saxl in \cite[Example~2.2]{Saxl}.
Given a partition $\lambda$, we define the \emph{rank} of $\lambda$
to be the maximum integer~$r$ such that~$\lambda_r \ge r$. (Thus
the partition $2[a_1,\ldots,a_r]$ has rank $r$.)
Let $\phi_k = \theta_k\ind^{S_{2k}}$. To prove the lemma,
we must show that $\phi_k = \sum \chi^{2[\alpha]}$, where the sum
is over all partitions $\alpha$ of~$k$ with distinct parts.

By
an easy application of Mackey's lemma (see \cite[Theorem~3.3.4]{Benson}) we have
\[ \phi_{k}\Res_{S_{2k-1}} = \phi_{k-1} \Ind^{S_{2k-1}}.\]
%\[\theta_k \Ind^{S_{2k}} \Res_{S_{2k-1}} = \theta_{k-1} \Ind^{S_{2k-1}}.\]
It follows by induction that
\begin{equation} \label{Eq:Phi}
\phi_k \Res_{S_{2k-1}} = \sum \chi^{2[\alpha]} \Ind^{S_{2k-1}}
\end{equation}
where the sum is over all partitions $\alpha$ of $k-1$ with distinct parts.
We now calculate, again using Mackey's lemma, that
\begin{align*}
\left< 1_{C_2 \wr A_k} \Ind^{S_{2k}},
1_{S_k} \times \sgn_{S_k} \Ind_{S_k \times S_k}^{S_{2k}} \right>  = 
\left< 1_{C_2 \wr A_k}\Ind^{S_{2k}}\Res_{S_k \times S_k}, 
1_{S_k} \times \sgn_{S_k} \right>_{S_k \times S_k} \\
= \sum_g \left< 1, 1_{S_k} \times \sgn_{S_k} \right>_{(C_2 \wr A_k)^g 
                            \cap (S_k \times S_k)}\hspace{1in} \\
   =   \sum_g 
	\left\{ \begin{array}{c@{\;\; : \;\;}l}1 &  (C_2 \wr A_k)^g \cap (S_k \times S_k) 
	\le S_k \times A_k \hspace{0.33in} \\
												0 & \hbox{otherwise}
									\end{array}\right.
										\\ \ge 1,\hspace{3.2in}
\end{align*}
where in the sums $g$ runs over a set of representatives for
the double cosets of $C_2 \wr A_k$ and $S_k \times S_k$
in $S_{2k}$. 
It follows from Pieri's rule that~$\phi_k$ 
contains either $\chi^{(k+1,1^{k-1})}$ or~$\chi^{(k,1^{k})}$
with positive multiplicity.
From~\eqref{Eq:Phi} we see that
the latter character cannot occur in $\phi_k$, while $\chi^{(k+1,1^{k-1})}$
can occur at most once. Thus 
$\phi_k$ contains $\chi^{2[k]}$ exactly once. 

Suppose now that $\chi^\la$ is a constituent of $\phi_k$. If $\la$ has rank $3$
or more, it follows immediately from \eqref{Eq:Phi} that $\la = 2[\alpha]$
for some $\alpha$. The rank $1$ and rank $2$ possibilities need a little
more care, but in close analogy with Saxl's argument, one can rule out the appearance
of any unwanted characters by using the known occurrence of $\chi^{2[k]}$.
Finally, suppose that $\alpha = (a_1, \ldots, a_r)$ 
is a partition of $k$
with distinct parts 
and that $\chi^{2[\alpha]}$ does not appear in~$\phi_k$. Then,
if~$\mu$ is the partition obtained from $2[\alpha_1, \ldots, \alpha_r]$ by
removing a node from row~$r$,
$\chi^{\mu}$ does not appear in $\phi_k\res_{S_{2k-1}}$, in contradiction
to \eqref{Eq:Phi}.$\qed$

\subsection{Corollaries}
Working through the cases in Theorem~\ref{thm:main}
we get a complete list of
all the irreducible characters of symmetric groups
that can be obtained as a constituent of
a multiplicity-free permutation representation.

%\vbox{
\begin{corollary}\label{cor}
Let $n \ge 66$ and let $\lambda$ be a partition of $n$. 
The irreducible character $\chi^\la$ is a constituent of
a multiplicity-free permutation character of~$S_n$ if and only if 
(at least) one of:

\begin{thmlist}
\item[\emph{(1)}] $\la$ has at most two rows or at most two columns;

\item[\emph{(2)}] $\la = (n-i,1^i)$
for some $i$ with $0 \le k < n$;

\item[\emph{(3)}] $\la=(2k-i,i,1)'$ where $n=2k$ and
$1 \le i \le k$;

\item[\emph{(4)}] $\la$ has at most one row of odd length;
%note this includes adding to (2n-2k,2k), but we missed the
%conjugates

\item[\emph{(5)}] $\la$ has columns all of even length and $n \equiv 2$ mod $4$;

\item[\emph{(6)}] $\la$ can be obtained by adding nodes to one of the following partitions
\[ (2,2,1), (2,2,2), (5,1^4), (4,4,1), (3,2,2,2) \]
subject to the restriction that all added nodes are in different columns;

\item[\emph{(7)}]
$\la$ can be obtained by adding nodes to one of the following partitions
\[ (3,2), (3,3), (2,2,1), (2,2,2), (5,1^4), (4,4,1), (3,2,2,2) \]
subject to the restriction that all added nodes are in different rows.\quad$\qed$
\end{thmlist}
\end{corollary} %}
\begin{proof}
Cases (1) and~(2) give the characters coming from 
cases (a1), (a2) and~(b1) of Theorem~\ref{thm:main}.
If $n = 2k+1$ then it follows
from the explicit calculations in \S 2.2 that
 the groups in case (b2)
contribute the further characters with labels $(2k-2i, 2i, 1)$, $(2k-2i,2i,1)'$
for $1 \le i \le k/2$
and $(2k-2i-1,2i+1,1)'$ for $0 \le i \le (k-1)/2$. The first family
is subsumed by case~(4); the others form case~(3).
The remaining cases are straightforward: case~(4) comes directly
from~(c), case~(5) from~(d) and cases~(6) and~(7) from~(e).
\end{proof}

%%Case (3) gives the characters from the index $2$ subgroups
%%of $S_k \wr S_2$ appearing in case (b1) which are not 
%%included in case (4).
%%Case (4) directly from (c), case (5) from (d),
%%and cases (6) and (7) from (e).

The following immediate corollary of Theorem~\ref{thm:main} is also
of interest.

\begin{corollary}
Let $n \ge 66$.
Suppose that  $G$ is a subgroup of $S_n$ such that the permutation
character of $S_n$ acting on the cosets of $G$ is multiplicity-free.
If the permutation character of $S_{n+1}$ acting on the cosets
of $G$ is also multiplicity-free then \emph{either}
$G \ge A_n$ \emph{or}
$n=2k$ is even and either
$G = S_2 \wr S_k$ or
$G$ is a transitive subgroup of $S_k \wr S_2$ of index 
at most $2$. \quad$\Box$
\end{corollary}

%Clearly not
%every irreducible character appears in this list. For example, if $n \ge 20$,
%then
%$\chi^{(n-4,3,1)}$ is not a constituent of a multiplicity-free permutation
%representation.

\section{Proof of Theorem~\ref{thm:mod}}
We begin by collecting 
the background results we need for the proof of Theorem~\ref{thm:mod}.
We shall distinguish between inner tensor products,
denoted~$\otimes$,
and outer tensor products, denoted~$\boxtimes$. 

\begin{lemma}\label{lemma:sgnfilt}
Let $F$ be a field.
If $U$ is a module for $FS_n$ with a Specht filtration
then $U \otimes sgn$ has a filtration by the duals of Specht modules.
In particular, 
if $M$ is a self-dual module for $FS_n$ with a Specht filtration,
then $M \otimes \sgn$ also has a Specht filtration.
\end{lemma}

\medskip
\noindent\emph{Proof.}
By  \cite[Theorem~8.15]{James}, if $\lambda$
is any partition then
\begin{equation}
\label{eq:modtwisted} 
S^{\lambda} \otimes \sgn \cong \bigl( S^{\lambda'} \bigr)^\star. 
\end{equation}
(This is the modular version of \eqref{eq:chartwist} above.)
Since the functor sending an~$FS_n$-module $U$ to $U \otimes \sgn$ is clearly exact, this 
is all we need
to prove the lemma.$\qed$

\medskip

\begin{theorem}\label{thm:JamesYoung}
Let $F$ be a field and let $n > k \ge 1$. If $\lambda$ is a partition
of~$k$ then
\[ S^\lambda \boxtimes F_{S_{n-k}} \Ind_{S_{k} \times S_{n-k}}^{S_n} \]
has a Specht filtration, with the Specht factors 
given by Young's rule. Similarly
\[ S^\lambda \boxtimes \sgn_{S_{n-k}} \Ind_{S_{k} \times S_{n-k}}^{S_n} \]
has a Specht filtration, with the 
Specht factors given by Pieri's rule. %\quad$\qed$
\end{theorem}

\noindent
\emph{Proof.}
The first statement 
follows from
James' modular version of Young's rule~\cite[Corollary~17.14]{James}.
The second may be deduced from the first by using~\eqref{eq:modtwisted}.
$\qed$

\medskip
Since the functor sending an $FS_k$-module $U$ to 
$U \boxtimes F_{S_{n-k}}\ind^{S_n}$ is exact, 
it follows from Theorem~\ref{thm:JamesYoung}
that if $U$ is an $FS_k$-module with a Specht filtration
then $U \boxtimes F_{S_{n-k}}\ind^{S_n}$ also has a Specht
filtration, with the Specht factors given by repeated
applications of Young's rule. Naturally there is a similar
result for $U \boxtimes \sgn_{S_{n-k}}\ind^{S_n}$.

It remains to state two results concerning
summands of permutation modules. Both of these have
a slightly technical flavour, but neither is at all difficult
to apply.

\begin{lemma}\label{lemma:selfdual} Let $F$ be a field
of prime characteristic $p$.
If $G$ is a subgroup of~$S_n$ such that the permutation
character $1_G \ind^{S_n}$ is multiplicity-free then
all the summands of $F \ind_G^{S_n}$ are self-dual.
\end{lemma}

\medskip
\noindent\emph{Proof.}
For simplicity, we assume that $F$ is the field with $p$ elements.
Let $U$ be an indecomposable direct summand of $F \ind_G^{S_n}$.
Let $\mathbf{Z}_p$
denote the ring of $p$-adic integers. 
Since~$U$ is a direct summand of a permutation module, we may
lift $U$ to a~$\mathbf{Z}_pS_n$-module~$V$ such that~$V$ 
is a direct summand of~$\mathbf{Z}_p\ind_G^{S_n}$
and~$V \otimes_{\mathbf{Z}_p} F = U$. 
(See \cite[\S 3.11]{Benson} for an outline of this
lifting process.)

Suppose that $U$ is not self-dual. Then the lifted module
$V$ is not self-dual either. Since $\mathbf{Z}_p \ind_G^{S_n}$
is self-dual, we may find a summand~$V'$
of~$\mathbf{Z}_p \ind_G^{S_n}$ such that $V' \cong V^\star$.
As $V$ and $V'$ are non-isomorphic, they are distinct 
summands of~$\mathbf{Z}_p \ind_G^{S_n}$.
But $V$ and $V^\star$ have the same 
ordinary character. This
contradicts our assumption that the character 
$1_G \ind^{S_n}$ is multiplicity-free.$\qed$

\medskip
This lemma deals with the assertions about
duality in Theorem~\ref{thm:mod}. It may
also be used to replace the reference
to the author's D.~Phil thesis \cite[Theorem~6.5.1]{WildonDPhil}
in the proof of Theorem~4 of \cite{PagetBrauer}.

Finally, we shall often be in the position of knowing that
a permutation module $F \ind_G^{S_n}$ 
has a Specht filtration, and
wishing to prove that the same result holds for 
each of its summands. Since in Theorem~\ref{thm:mod} 
we assume that our
ground field~$F$ is algebraically closed
and of characteristic $> 3$, 
we may use the homological algebra approach developed by 
Hemmer and Nakano in~\cite{HN}. (For an alternative, slightly
less technological approach, see the remark attributed to~S.~Donkin at the end of \S1 of \cite{PagetBrauer}.)

\vbox{
\begin{proposition}\label{prop:summands}
Let $F$ be an algebraically closed field of prime characteristic $p > 3$.
Let $M$ be a module for $FS_n$
with a Specht filtration. If $U$ is a direct
summand of $M$ then $U$ also has a Specht filtration.
\end{proposition}

\noindent
\emph{Proof.}
This is immediate from~\cite[Theorem~3.6.1]{HN}.$\qed$
}

\medskip
We are now ready to prove Theorem~\ref{thm:mod}. By the
work of Hemmer and Nakano \cite{HN}, when $p > 3$, the multiplicities
of the factors in a Specht filtration are well-defined. Hence
it suffices to show that each of the permutation modules
in Theorem~\ref{thm:mod} has a Specht filtration, with the 
Specht factors given by its ordinary character.

We start
with the modules coming from case~(a) of Theorem~\ref{thm:main}. 
Suppose that $G = A_k \times A_{n-k}$.
Since the ground field~$F$ 
has odd characteristic,
\[ \begin{split} F \Ind_{A_k \times A_{n-k}}^{S_n} = 
\bigl( F_{S_k} \boxtimes F_{S_{n-k}} \bigr) \Ind^{S_n}
\oplus 
\bigl( \sgn_{S_k} \boxtimes F_{S_{n-k}} \bigr) \Ind^{S_n} \qquad\qquad\qquad\qquad \\
\qquad\qquad\oplus
\Bigl( \bigl( F_{S_k} \boxtimes F_{S_{n-k}} \bigr) \Ind^{S_n} \otimes \sgn_{S_n} \Bigr)
\oplus 
\Bigl( \bigl( \sgn_{S_k} \boxtimes F_{S_{n-k}} \bigr) \Ind^{S_n} \otimes \sgn_{S_n}
\Bigr).
\end{split}\]
It follows from Lemma~\ref{lemma:sgnfilt}
and Theorem~\ref{thm:JamesYoung} that
each of the four summands 
has a Specht filtrations. Proposition~\ref{prop:summands}
now guarantees that \emph{any} indecomposable summand 
of~$F \ind_{A_k \times A_{n-k}}^{S_n}$ has a Specht filtration.
This deals with all the subgroups appearing in case (a),
and also the subgroups of $S_{k+1} \times S_k$ in case (b2).
%, and also
%all the subgroups of~$S_k \times S_k$ that appear
%in case (b).

Now suppose that $n=2k$. We first note that
\[ F\Ind_{S_k \times S_k}^{S_{2k}} = F\Ind_{S_k \wr S_2}^{S_{2k}}
\oplus \psi_k \Ind_{S_k\wr S_2}^{S_{2k}}\] 
where $\psi_k$ is the $1$-dimensional representation of
$S_k \wr S_2$ defined in \S 2.2.
Hence both of the summands have a Specht filtration.
(Here it is essential
that~$F$ has odd characteristic: see \S 4.2 for
an example when $F$ has characteristic~$2$.)
It follows that
\[ F\Ind_{A_{2k} \cap (S_k \wr S_2)}^{S_{2k}}
= F \Ind_{S_k \wr S_2}^{S_{2k}} \oplus \left( F \Ind_{S_k \wr S_2}^{S_{2k}} \right)
\otimes \sgn_{S_{2k}}
\]
and
\[ F\Ind_{\Gamma_k}^{S_{2k}} = 
 F \Ind_{S_k \wr S_2}^{S_{2k}} \oplus 
 \left( \psi_k \Ind_{S_k \wr S_2}^{S_{2k}} \right) \otimes \sgn_{S_{2k}}\]
both have Specht filtrations. 
Moreover, 
%it
%is straightforward to show
%that 
if $G$ is a subgroup of $S_k \wr S_2$ of index~$4$ then
$F \ind_G^{S_{2k}}$ is in every case
a direct summand of $F\ind_{A_k\times A_k}^{S_{2k}}$,
and so has a Specht filtration.
Hence if $G$ is any subgroup
of $S_{2k}$ such that
 $A_k \times A_k < G \le S_k \wr S_2$, then $F\ind_G^{S_{2k}}$
has a Specht filtration. By
Theorem~\ref{thm:JamesYoung}, %we see that 
$F \ind_{G}^{S_{2k+1}}$
also has a Specht filtration. These remarks deal with
all the remaining subgroups in case (b).

In cases (c) and (d), a considerable amount of work is done for us by
Theorem~2 of \cite{PagetBrauer}, which implies that
$F\ind_{S_2 \wr S_k}^{S_{2k}}$ has a Specht filtration. 
By
Theorem~\ref{thm:JamesYoung},
$F\ind_{S_2 \wr S_k}^{S_{2k+1}}$ also has a Specht filtration. 
The argument giving~\eqref{eq:collide} in~\S 2.4 shows that
\[ F_{H \rtimes S_k}\Ind^{S_{2k}} = 
F_{S_2 \wr S_k}\Ind^{S_{2k}} \oplus \bigl( F_{S_2 \wr S_k}\Ind^{S_{2k}} 
\bigr) \otimes \sgn_{S_{2k}}.\]
We may now apply 
Lemma~\ref{lemma:sgnfilt} and Proposition~\ref{prop:summands}
to deduce that the summands of the left-hand side have Specht filtrations.

%To complete the proof of Theorem~\ref{thm:mod} 
It only remains to deal with the permutation modules
from case~(e). By our usual arguments, together with the remark following
Theorem~\ref{thm:JamesYoung}, it suffices to show
that for $k = 5$,~$6$ or~$9$, $F\ind_{G_k}^{S_k}$ has
a Specht filtration. This is immediate if $p > k$, as then
every $FS_k$-module is a direct sum of Specht modules.
The other cases turn out to be surprisingly easy.

If $p = 5$ and $k=5$ or $k=6$ then
the Specht modules corresponding to the non-trivial ordinary characters in
$1_{G_k}\ind^{S_k}$ are simple and projective. Hence
\[ F\Ind_{G_5}^{S_5} \cong S^{(5)} \oplus S^{(2,2,1)} \quad\text{and}\quad
F\Ind_{G_6}^{S_6} \cong S^{(6)} \oplus S^{(2,2,2)}.\]
When $p = 5$ and $k=9$ we observe that since $\PGammaL(2,8)$
does not contain any elements of order $5$,~$F\ind_{G_9}^{S_9}$ is projective.
It is well known (see \cite{JamesYoung}) that any projective module for
a symmetric group has a Specht filtration. The only
case left is when $p = 7$ and $k=9$.
One sees from~\eqref{eq:PGammaL} 
that each irreducible ordinary character appearing
in $1_{G_9}\ind^{S_9}$ lies in a different $7$-block of~$S_9$, and that
the only non-projective constituents are the trivial and sign
representations. Hence
\[ F\Ind_{G_9}^{S_9} \cong S^{(9)} \oplus S^{(1^9)} \oplus
S^{(4,4,1)} \oplus S^{(3,2,2,2)} \oplus S^{(5,1^4)} .\]
This completes the proof of Theorem~\ref{thm:mod}.

%\newpage
\section{Two counterexamples}

\subsection{} 
Let $F$ be an algebraically closed field of characteristic $3$.
We shall show
that the module  $M = F \ind_{\PGammaL(2,8)}^{S_9}$
is a counterexample to the conjecture that every permutation
module has a Specht filtration. 
It would be interesting to collect further examples
of permutation modules over fields
of odd characteristic which do not have 
Specht filtrations---at the moment, 
it is far from clear how common they are.

The shortest demonstration
that the author has been able to discover
hinges on the simple module $D^{(5,2,2)}$.
(See \cite[Definition~11.2]{James}
for the definition of the $D^\mu$.)
One can show, either with the help of computer algebra,
or more lengthily by \emph{ad hoc} arguments (see below), 
that $M$ contains a submodule isomorphic to
$D^{(5,2,2)}$. It follows from the table
of decomposition numbers of~$S_9$
in characteristic~$3$ (see \cite[p145]{James})
that there are only four Specht modules which (a)
have $D^{(5,2,2)}$ as a composition factor, and (b) do not also
have other composition factors that are absent in~$M$. They are
\[ S^{(5,2,2)}, S^{(3,2^3)}, S^{(5,1^4)}, S^{(3,3,1^3)}.\]
By a standard result (see \cite[Corollary~12.2]{James}),
$D^{(5,2,2)}$ only appears at the top of $S^{(5,2,2)}$.
By~\eqref{eq:modtwisted},
$D^{(5,2,2)}$ appears in the socle of  $S^{(3,2^3)}$ if and only if
$D^{(5,2,2)} \otimes \sgn = D^{(6,2,1)}$ appears in the top of
$S^{(4,4,1)}$; this is also ruled out by
\cite[Corollary~12.2]{James}. The same argument works
for $S^{(3,3,1^3)}$. Finally, one can use the long exact sequence
\[ S^{(9)} \rightarrow S^{(8,1)} \rightarrow
S^{(7,1^2)} \rightarrow S^{(6,1^3)} \rightarrow S^{(5,1^4)} \rightarrow \cdots\]
given by the maps $\theta_r$ constructed by Hamernik in \cite[p449]{Hamernik}
to show that~$S^{(5,1^4)}$ contains $D^{(5,2,2)}$ in its top,
but not in its socle. (Hamernik works only with symmetric groups
of prime degree, but it is easy to generalise this part of his work
to deal with hook-Specht modules for $FS_n$ whenever the
characteristic of~$F$ 
divides~$n$: see~\cite[\S 1.3]{WildonDPhil}.) 
Thus none of the candidate Specht modules contains $D^{(5,2,2)}$
as a submodule. The result follows.

It remains to show that $D^{(5,2,2)}$ appears in the socle
of $M$. For this we shall need the following lemma, which
is of some independent interest.

\begin{lemma}\label{lemma:endo}
Let $F$ be an algebraically closed field, let $G$
be a finite group and let $M$
be an indecomposable $FG$-module such
that 
\begin{thmlist}
\item $\soc M \cong \Top M \cong F$,
\item $F$
appears exactly twice as a composition factor of $M$.
\end{thmlist}
Then $\End_{FG}(M)$ is $2$-dimensional.
\end{lemma}

\begin{proof}
Let $\theta \in \End_{FG}(M)$.
By a corollary of Fitting's Lemma (see \cite[Lemma~1.4.5]{Benson}),
$\theta$ is either nilpotent or invertible. 
If $\theta$ is nilpotent then
$M / \ker \theta$ is isomorphic to a proper submodule of $M$;
by~(b) this can only happen if $\ker \theta = \rad M$ and~$\theta$ is,
up to a scalar, the map $\nu$ defined by
\[ M \twoheadrightarrow M/\rad M  \cong F \hookrightarrow M. \]
If $\theta$ is invertible, then it has a
non-zero eigenvalue $t \in F$. Since $\theta - t1_M$ is not
invertible, it must be a scalar multiple of $\nu$.
Hence $\End_{FG}(M) = \left<1_M, \nu\right>_F$. % is $2$-dimensional.
\end{proof}

\begin{proposition}\label{lemma:522} 
Let $F$ be an algebraically closed field of characteristic $3$.
The simple module $D^{(5,2,2)}$ 
lies in the socle of $F \ind_{\PGammaL(2,8)}^{S_9}$. 
\end{proposition}

%\medskip
\noindent\emph{Proof}. By basic Clifford theory,
it is equivalent to show that $D^{(5,2,2)}\res_{A_9}$
lies in the socle of $N = F \ind_{\PGammaL(2,8)}^{A_9}$; for an introduction
to the Clifford theory needed to relate representations of 
the alternating and symmetric groups, see \cite[Chapter~5]{FH}.
It follows from \eqref{eq:PGammaL} in \S 2.3 that
\[ 1_{\PGammaL(2,8)}\Ind^{A_9} = 1 + %\chi^{(9)}\Res_{A_9} 
{\chi^{(5,1^4)}}^+ 
+ \chi^{(4,4,1)}\Res_{A_9}\]
where ${\chi^{(5,1^4)}}^+$ is one of the two irreducible constituents of 
$\chi^{(5,1^4)}\res_{A_9}$. (The labelling of this pair of characters
is essentially arbitrary, and we do not need to be any more precise here.) 
From this, one can use decomposition 
numbers of $S_9$ to show that the composition factors of
%one can use a 
%table of decomposition numbers of $S_9$
%in characteristic~$3$ to
%show that the composition factors of 
$N$ % = F \ind_{\PGammaL(2,8)}^{A_9}$
are
\[ F, F, D^{(8,1)}\Res_{A_9}, D^{(6,3)}\Res_{A_9},
D^{(5,2,2)}\Res_{A_9}, D^{(5,2,2)}\Res_{A_9} .\]
%Note that  $D^{(5,2^2)} = D^{(6,2,1)} \otimes \sgn$,
%$D^{(4^2,1)} = D^{(8,1)} \otimes \sgn$,
%and $D^{(3^2,2,1)} = D^{(6,3)} \otimes \sgn$; 

Let $Q$ be a Sylow $3$-subgroup of $\PGammaL(2,8)$ and let $P$
be a Sylow $3$-subgroup of $A_9$ containing $Q$. Note that $|P : Q| = 3$.
It follows from Mackey's lemma that
\begin{equation}\label{eq:dec} 
%\mathbf{F}_3\ind_{\PGammaL(2,8)}^{A_9} 
N\Res_P = \bigoplus_g
F\Ind_{\PGammaL^g \cap P}^{P}
\end{equation}
where $g$ runs over a set of representatives for the double cosets
of $\PGammaL(2,8)$ and $P$ in $A_9$. If $R$ is any subgroup of
$P$ then, by Frobenius reciprocity
\[ \dim \Hom_{FP}(F\ind_R^{P}, F)
= \dim \Hom_{FR}(F,F) = 1. \]
Hence each summand of the right-hand side of \eqref{eq:dec} is
indecomposable with a $1$-dimensional socle. The dimension
of $F\ind_{\PGammaL^g \cap P}^{P}$ is
$|P : \PGammaL^g \cap P|$ which is divisible by $|P : Q| = 3$.
It follows that if $U$ is any indecomposable
summand of $N$ (considered as a $FA_9$-module) then
$U$ has dimension divisible by $3$.

%We now gather some information about the endomorphism
%algebra of $N$, aiming to apply Lemma~\ref{lemma:endo}.
Since the ordinary character $1_{\PGammaL(2,8)}\ind^{A_9}$
is multiplicity-free with three irreducible constituents, the 
$\mathbf{Q}A_9$-permutation module $\mathbf{Q}\ind_{\PGammaL(2,8)}^{A_9}$
has a $3$-dimensional endomorphism algebra.
By \cite[Theorem~3.11.3]{Benson}, the endomorphism algebra
$\End_{FA_9} (N)$
is also $3$-dimensional. 

The $7$-dimensional
simple module $D^{(8,1)}\res_{A_9}$ and the $41$-dimensional
simple module $D^{(6,3)}$ cannot appear as summands of $N$.
Since they each appear but once as compositions factors
of~$N$, they must lie in its middle
Loewy layer. If, in addition, $D^{(5,2,2)}\res_{A_9}$ does not
appear in the socle of $N$, then~$N$ must be indecomposable,
with top and socle both isomorphic to the trivial module~$F$.
But then Lemma~\ref{lemma:endo} implies that $\End_{FA_9}(N)$
is $2$-dimensional, a contradiction. Hence $\soc N$
contains $D^{(5,2,2)}\res_{A_9}$.
%$\soc N \cong \Top N \cong D^{(5,2,2)}\res_{A_9}$, as required.
\quad$\Box$

\medskip
\noindent \emph{Remark:}
A small extension of this argument shows that 
$M = F\ind_{\PGammaL(2,8)}^{S_9}$ 
is indecomposable, with the Loewy layers
shown below. 
\[ \begin{matrix} F \oplus \sgn \oplus D^{(6,2,1)}
\oplus D^{(5,2,2)} \\
D^{(8,1)} \oplus D^{(4,4,1)} \oplus D^{(6,3)} \oplus D^{(3,3,2,1)}
\\
F \oplus \sgn \oplus D^{(6,2,1)}
\oplus D^{(5,2,2)} \end{matrix}
\]

\subsection{}
We now consider $\mathbf{F}_2 \ind_{S_3 \wr S_2}^{S_6}$.
It is easy to show that this module has the Loewy
layers
\[ 
%\mathbf{F}_2 \ind_{S_3 \wr S_2}^{S_6} = 
\begin{array}{c} \mathbf{F}_2 \\ D^{(5,1)} \oplus D^{(4,2)} \\ \mathbf{F}_2
\end{array}.
\]

By \eqref{eq:twowreath} in \S2.2, the ordinary character associated to
 $\mathbf{F}_2 \ind_{S_3 \wr S_2}^{S_6}$ 
is $\chi^{(6)} + \chi^{(4,2)}$.
It is known that although the trivial module is a 
composition factor of~$S^{(4,2)}$, it does not appear
in either the top or socle of $S^{(4,2)}$ (see 
\cite[Example~24.5(iii)]{James}).
It follows that there is no
Specht filtration of $\mathbf{F}_2 \ind_{S_3 \wr S_2}^{S_6}$ 
with
the factors $S^{(6)}$ and $S^{(4,2)}$. 
One can, however, exploit the outer automorphism of $S_6$, which
sends the Specht module $S^{(5,1)}$ to 
$S^{(3,3)} = {S^{(2,2,2)}}^\star$ and leaves~$\mathbf{F}_2\ind_{S_3 \wr S_2}^{S_6}$ fixed,
to show that there is a short exact sequence
\[ 0 \rightarrow S^{(5,1)} \rightarrow 
\mathbf{F}_2 \ind_{S_3 \wr S_2}^{S_6} \rightarrow
S^{(2^3)} \rightarrow 0 .\]
Thus $\mathbf{F}_2 \ind_{S_3 \wr S_2}^{S_6}$
has a Specht filtration, but the Specht factors 
required are not those indicated by
the associated ordinary character. It is left to the
reader to formulate any of the many conjectures to
which this module is a counterexample.

%Finally we remark that if $p = 2$ then similar arguments
%to those just use show that
%$\mathbf{F}_2 \Ind_{\PGammaL(2,8) \times S_{n-9}}^{S_9}$ does
%not have a Specht filtration. 

%

\section*{Acknowledgements}
I should like to thank Michael Collins
for asking me the question that lead to the work of \S 2,
and an anonymous referee for his or her
very helpful comments on earlier versions of this paper.

\def\cprime{$'$} \def\Dbar{\leavevmode\lower.6ex\hbox to 0pt{\hskip-.23ex
  \accent"16\hss}D}

\end{document}